\documentclass{amsart}

\usepackage{graphicx}

\theoremstyle{plain}
\newtheorem{theorem}{Theorem}[section]
\newtheorem{lemma}[theorem]{Lemma}

\newtheorem{corollary}[theorem]{Corollary}
\theoremstyle{definition}
\newtheorem{definition}[theorem]{Definition}

\newtheorem{example}[theorem]{Example}
\theoremstyle{definition}
\theoremstyle{remark}
\newtheorem{remark}{Remark}

\makeatletter
\def\@seccntformat#1{\csname the#1\endcsname.\quad}
\def\section{\@startsection{section}{1}{\z@}%
{-3.5ex \@plus -1ex \@minus -.2ex}%
{2.3ex \@plus.2ex}%
{\Large\bf}}
\def\subsection{\@startsection{subsection}{2}{\z@}%
{-3.25ex\@plus -1ex \@minus -.2ex}%
{1.5ex \@plus .2ex}%
{\bf\large}}
\def\subsubsection{\@startsection{subsubsection}{3}{\z@}%
{-3.25ex\@plus -1ex \@minus -.2ex}%
{1.5ex \@plus .2ex}%
{\normalfont\normalsize\bf}}

\def\blfootnote{\xdef\@thefnmark{}\@footnotetext}
\makeatother

\newcommand{\R}{\mathbb{R}}

 \newcommand{\be}{\begin{equation}}
 \newcommand{\ee}{\end{equation}}
 \newcommand{\bd}{\begin{displaymath}}
 \newcommand{\ed}{\end{displaymath}}
 \newcommand{\bea}{\begin{eqnarray}}
 \newcommand{\eea}{\end{eqnarray}}
 \newcommand{\beas}{\begin{eqnarray*}}
 \newcommand{\eeas}{\end{eqnarray*}}
 \newcommand{\bc}{\begin{center}}
 \newcommand{\ec}{\end{center}}

 \newcommand{\pa}{\partial}

 \def \O{\Omega}

   \def \S{\Sigma}
 \title{Numerical approximation of one phase quadrature domains}

\author{Mahmoudreza Bazarganzadeh}
\address{Department of Mathematics, Stockholm University,
S-10691, Stockholm, Sweden}
\email{mahmoudreza@math.su.se}
\thanks{M. Bazarganzadeh thanks Stockholm University for supporting to  visit IST, Lisbon.}

\author{Farid Bozorgnia}
\address{Department of Mathematics, Instituto Superior T\'{e}cnico, Lisbon.}
\email{bozorg@math.ist.utl.pt}
\thanks{F. Bozorgnia  was supported by the UT Austin-Portugal partnership through
the FCT post-doctoral fellowship
SFRH/BPD/33962/2009 and grants PTDC/MAT/114397/2009, UT Austin/MAT/0057/2008}

\begin{document}
\begin{large}
\vskip12mm

\maketitle

\renewcommand{\theequation}{1.\arabic{equation}}
\setcounter{equation}{0}

\begin{abstract}
In this work, we present two numerical schemes for    a  free boundary problem called one phase quadrature domain. In the first method by
applying the proprieties of given free boundary  problem, we derive a  method that leads to a fast iterative  solver. The iteration procedure is adapted in order to
work in the case when  topology changes.   The  second method is based on shape reconstruction 
to establish an efficient Shape-Quasi-Newton-Method.
   ~Various numerical experiments confirm the efficiency
of the derived numerical methods.


\end{abstract}
\tableofcontents

\section{Introduction.}

In this paper we shall consider general mathematical  approaches to solve the free boundary problems of type
\begin{equation}\label{shape11}
A(u, \Omega) = 0,
\end{equation}
\begin{equation}\label{shape2}
B(u, \Omega) = 0.
\end{equation}
Here $A$ corresponds to a well posed elliptic boundary value problem in an unknown domain  $\Omega=\{x\,:\,u(x)>0\}=\{u>0\}$ and $B$  operates on the
functions supported at the free boundary $\Gamma=\pa \O$. It is supposed that function $u$ can be solved from equation (\ref{shape11}) for any given
suitable domain  $\Omega.$  More precisely, in this paper we consider the following problem:
\be\label{one phase for sub1}
 \begin{cases} \tag{P}
 \Delta u= \chi _{\O} -\mu, & \mbox{in}  \quad\,\, \R^N, \\
 u\geq 0, &\mbox{in} \quad \R^N, \\
u=0 ,&\mbox{in} \quad \R^N\setminus \Omega,
 \end{cases}
\ee
where $\mu$ is a given measure with compact support. Our aim in this work  is to study  systematic and efficient ways to solve Problem (P) numerically. The outline of the paper is as follows.

 In section 2, we present  some basic facts and mathematical background of quadrature domains. In section 3, we
investigate one of the applications of  quadrature domains, Hele Shaw flow. Section 4 is devoted to derive a numerical scheme which is based on the properties of the free boundary especially blow up techniques. 
In section 5 we  construct another  numerical scheme for Problem (P) based on shape reconstruction formulation.
Finally, in last section we investigate some numerical examples which show the efficency of the numerical algorithms.

\section{Notations and mathematical background of quadrature domains.}
\renewcommand{\theequation}{2.\arabic{equation}}
\setcounter{equation}{0}
Let us review some notations that we use here. By $\Omega$ we mean an open subset of $\R^N$ and $L^p(\O)$ the usual Lebesgue space with respect to the Lebesgue measure.  $HL^p(\O)$ denote the subspace of $L^p(\O)$ that consists of harmonic functions and $SL^p(\O)$ for the subspace of $L^p(\O)$ that consists of
subharmonic functions. We also show the characteristic function of $\O$ by $\chi_\O$ and $G$ always denotes the usual ''fundamental solution'' for the Laplace
operator in $\R^N$. In other words, for $x \in \R^N\setminus \{0\}$
\beas
  G(x)=  \begin{cases}
  \frac{1}{N(N-2)\omega_N} |x|^{2-N},&{\rm for ~~~} N \geq 3,\\
  -\frac{1}{2\pi} \ln |x|, &{\rm  for ~~~} N = 2,
  \end{cases}
\eeas
where $\omega_N$ denotes the volume of the unit ball in $\R^N$.

It is immediately verified that  if $\O$ is open and bounded then as function of $x\in \O$,
\beas\label{sub ineq}
&&G(x-y)\in HL^1(\O),\quad\,\, \text{for all } y\in \O^c,\\
&&-G(x-y)\in SL^1(\O),\quad\,\, \text{for all } y\in \O.
\eeas
\begin{definition}
Let  $\mu$ be a measure with compact support. By a \textit{subharmonic quadrature domain}  we mean an open connected set $\O\subset \R^N$   such that $supp(\mu)\subset \O$ and
\be\label{qi}
\int_\O h\,dx\geq\int\,h\,d\mu,
\ee
holds for all $h\in SL^1(\O)$. We write $\O\in Q(\mu,SL^1)$  if (\ref{qi}) holds and $\mu(\O)<\infty$.
\end{definition}

If one consider $\int_\O h\,dx=\int\,h\,d\mu$ for all $h\in HL^1(\O)$ then $\O$ is a \textit{quadrature domain} and we  write $\O\in Q(\mu,HL^1)$.

 The simplest quadrature domain is a circular disc. Suppose that  $\mu=\alpha \delta_0$ where $\delta_0$ is a  Dirac mass at origin and $\alpha>0$. Then
 $$Q(\mu,HL^1)=Q(\mu,SL ^1)=\{B(0;r)\},$$
 where $r\geq 0$ is determined by $|B(0;r)|=\alpha$, (see \cite{bjorn1}).
  Generally if $\O$ is a bounded domain in $\R^N$ and
\begin{equation}\label{mean value}
 \int_{\O}~h dx= |\O|\,h(x_0),
\end{equation}
 holds for all $h\in HL^1(\O)$, where $x_0$ is an arbitrary point, then $\O$ is a ball centered at $x_0$.

 In this article we  deal only with subharmonic quadrature domain, and from now on by a quadrature domain we mean a subharmonic quadrature domain.

Let  $U^\mu$,  the Newtonian
 potential of the measure $\mu$ which is defined by
 \beas
 U^\mu(x) := (G*\mu)(x) = \int_{\R^N} G(x-y) d\mu(y),~~~x \in \R^N,
 \eeas
and it satisfies the Poisson's equation $-\Delta U^\mu=\mu$ in the distribution sense. For the sake of simplicity, we shall use $U^{\O}$ instead of $U^{\chi_\O}
$.

 Gustafsson in  \cite{bjorn1} has showed that  $\Omega \in Q(\mu, SL^1) $ if and only if
\be\label{q1} \begin{cases}
  U^{\Omega}  \leq  U^\mu, & \hbox{in}\,\R^N,\\
   U^{\Omega}  =    U^\mu, & \hbox{in}\, \R^N \setminus \Omega.
  \end{cases}
 \ee
 Also if one considers $u = U^\mu - U^{\Omega}\geq 0$, then
\be \label{1-phase-1}
 \Delta u= \chi _\Omega -\mu \quad  \mbox{in}  \quad \R^N.
 \ee

Note that from (\ref{1-phase-1}) one has $\Delta u= \chi _\Omega$ away from $supp(\mu)$ and according to the results in local regularity of solutions for elliptic
PDEs, we obtain $u\in W_{\text{loc}}^{2,p}(\R^N)$,  for every $1<p<\infty$. Also $\nabla u\in W_{\text{loc}}^{1,p}(\R^N)$. By the Sobolev embedding theorem the first derivatives
are therefore H\"{o}lder continuous with  H\"{o}lder exponent $\alpha<1$. 

Sakai in \cite{sa2} has  proved that the  definition of quadrature domain is equivalent to the well-known one-phase free  boundary problem in distribution scenes. More
precisely, from PDE point of view, $\O\in Q(\mu,SL^1)$ is equivalent to
 \be\label{one phase for sub}
 \begin{cases} \Delta u= \chi _\Omega -\mu, &\mbox{in}  \quad \R^N, \\
 u\geq 0,& \mbox{in} \quad \R^N,\\
u=0,& \mbox{in} \quad \R^N\setminus \Omega.
 \end{cases}
\ee

\begin{remark}\label{rem1} Suppose that $m$ denotes the Lebesgue measure. By (\ref{q1}) we know that $u=|\nabla u|=0$ in $\R^N\setminus \O$. Now taking  integration of (\ref{1-phase-1}) gives
$$0=\int_{\pa \O}\frac{\pa u}{\pa n}\, ds=\int_{\O}\Delta u\, dx\,=\,m(\O)-\mu(supp(\mu)).$$
This fact is also a consequence of (\ref{qi}). In other words, we know the volume of the solution priori.
\end{remark}

\begin{example} \label{example} As an other example of one phase quadrature domain, suppose that $x_0\in \R^N$ and $a>0,\,M>1$.
Let $B_1=B_ 1(x_0,a)$ and $\mu=M \chi_{B_1}$. Then  (\ref{one phase for sub}) reads
  \be\label{example 2 pde}
  \begin{cases}
  \Delta u=\chi_{\Omega}-M\chi_{B_1}, &\mbox{in}\,\, \Omega,\\
  u=|\nabla u|=0,& \mbox{in}\,\, \Omega^c.
  \end{cases}
  \ee

The spherical symmetry of the problem shows that the we have to find a radial solution $u=u(|x|)$ for (\ref{example 2 pde}). Consequently, we suppose that $\Omega=B_
2=B_ 2(x_0,r)$ for some $r>a$. To make more easier we consider that $u=u_ 1$ on $B_1$ and $u=u_ 2$ on $B_2\setminus B_1$. We desire to patch $u_ 1$ and $u_ 2$
 on $\pa B_1$ without loosing regularity. Therefore our problem is
 \be
 \Delta u = \begin{cases}
1-M,& \mbox{in}~~B_1,\\
 1,& \mbox{in}~~B_2\setminus B_1 ,
\end{cases}
\ee
with the following conditions
 \be
 \begin{cases}\label{conditions}
 u_ 1=u_ 2,\,\,\nabla u_ 1=\nabla u_ 2,& \mbox{in}~~B_1,\\
  u_ 2=|\nabla u_ 2|=0& \mbox{in}~~(B_2)^c.
   \end{cases}
\ee
 By some calculations and considering the fundamental solution for Laplacian operator one has
 \be
 u(x) = \begin{cases}
 (1-M)\frac{|x-x_0|^2}{2N}+A_1,& \mbox{in}~~B_1,\\
 \frac{|x-x_0|^2}{2N}+A_2|x-x_0|^{2-N}+A_3,& \mbox{in}~~B_2\setminus B_1,\\
 0,& \mbox{in}~~(B_2)^c,
 \end{cases}
 \ee
 where $A_1, A_2, A_3$ are appropriate constants which are computed with respect to (\ref{conditions}). Now we obtain
 \be\label{solution exa2}
 u(x) =  \begin{cases}
 (1-M)\frac{|x-x_0|^2}{2N}+\frac{r^2-M a}{2(2-N)},& \mbox{in}~~B_1,\\
 \frac{|x-x_0|^2}{2N}-\frac{r^N|x-x_0|^{2-N}}{N(N-2)}+\frac{r^2}{2(2-N)},& \mbox{in}~~B_2\setminus B_1,\\
 0,& \mbox{in}~~(B_2)^c,
 \end{cases}
 \ee
 where $r=M^\frac{1}{N}a$.

 For $N=2$ we can replace $|x-x|^{2-N}$ by $\log |x-x_0|$ in (\ref{solution exa2}) and we derive that
 \be\label{solution exa2 in 2 dim}
 u(x) = \begin{cases}
 (1-M)\frac{|x-x_0|^2}{4}+\frac {a^2}{4}M+\frac{r^2}{2}(\log \frac{a}{r}-\frac{1}{2}\big),& \mbox{in}~~B_1,\\
 \frac{|x-x_0|^2}{4}+\frac{r^2}{2}\log|x-x_0|-\frac{r^2}{2}\big(\log r+\frac{1}{2}\big),& \mbox{in}~~B_2\setminus B_1,\\
 0,& \mbox{in}~~(B_2)^c,
 \end{cases}
 \ee
 with $r=M^\frac{1}{2}a$.
\end{example}

\begin{remark}
Suppose that $N=2,\, \mu=\delta_0$.  Let $B(0,\epsilon)$  be an approximation of $supp(\mu)$ with   $M=\frac{1}{\pi \epsilon^2}$ for $\epsilon$ small enough. Then one can obtain $r=\frac{1}{\sqrt{\pi}}$.
\end{remark}

\subsection{An estimate of quadrature domain.}
 In Problem (P) the domain $\Omega$  is  part of the solution and in order to generate a mesh, one needs to find a domain which  contains $ \Omega$. To do this  we
 find a bigger domain such that $\Omega$ is embedded in it as follows.

 By $r(\mu)$ we mean a positive number
corresponding to the positive measure $\mu$ such that $$|B_{r(\mu)}|=m(B_{r(\mu)})=\int_{\R^N }\,d\mu=\mu(\R^N),$$
where $m$ denotes the Lebesgue measure in $\R^N$. The following theorem is due to Sakai, see \cite{sa3}.

\begin{theorem}\cite{sa3}\label{estimate} Let $\mu$ be a finite positive measure with support in the closed  ball $\overline{B_R},\,R>0$. Then every quadrature domain $\O$
of
$\mu$ for subharmonic functions satisfies \be \O\subset B_{r(\mu)+R}. \ee Furthermore,  if $r(\mu)> 2R$ then
\beas
B_{r(\mu)-R}\subset \O.
\eeas \end{theorem}

For instance let $\mu=g(x)\chi_{B_1}$, where $g$ is a positive function with $M=\sup_{B_1}g(x) $, then $\O\subset B_{\sqrt{M}+1}$.

For more information about one phase quadrature domain see \cite{bjorn1, gp, gs, sa2, sakai1}.

\section{An application (Hele Shaw flow).}
\renewcommand{\theequation}{3.\arabic{equation}}
\setcounter{equation}{0}
 One application of Problem (P) appears in Laplacian growth like Hele-Shaw flow which comes up in flow's dynamic. Here we describe the Hele-Shaw problem briefly.

Suppose that some incompressible fluid has been confined between two parallel plate and we inject more fluid by moderate velocity to it. Therefore the fluid between plates begin to occupy more
 space. We are interested in to study the behavior of the boundary of the fill space.

To be precise,  let $\nu$ be a positive, finite and non zero measure with compact support and $supp(\nu) \subseteq
D$ where $D$ is an open subset of $\R^N$  by a $C^1$ boundary. 
Let  $p_D$ be a super harmonic function such that
\begin{equation}\label{hele shaw}
\begin{cases}
 -\Delta p_D= \nu & \hbox{in}\, D,\\
p_D=0 & \hbox{on} \, \partial D.
\end{cases}
\end{equation}
We are looking for a family of regions $D_t$ for $t\geq 0$, such that $\partial D_t$  moves with the velocity  $-\nabla p_{D_t}$. This problem was introduced by S. Richardson \cite{r}.

\begin{definition} Suppose that $I$ is an interval in $\R$. Let $\mu=\chi_{D_0}+t\nu$, $t\in I$. A map $ t\rightarrow D_t\subset \R^N$ is a \textit{weak
solution} of the free boundary problem if the function $u_t\in H^1{(\R^N)}$ defined by
\begin{equation}\label{weak solution}
\Delta u_t=\chi_{D_t}-\mu,
\end{equation}
satisfies the following conditions:

\begin{itemize}\label{hele shaw2}
\item $u_t\geq 0,$

 \item $\int u_t(1-\chi_{D_t})\,dx=0.$

 \end{itemize}
 Last condition guarantee that $u_t=0$ in $\R^N\setminus D_t$, see \cite{gus1}.
\end{definition}

\begin{remark}  PDE (\ref{weak solution}) together with the above conditions  are a special case of Problem (\ref{one phase for sub1}).
\end{remark}
Next theorem states the corresponding quadrature domain  of the solution of the Hele-Shaw problem.
\begin{theorem}\cite{gus1}
Suppose that $\mu$ and $D_0$ are as before and $T>0$. Then there exists a weak solution
$$[0,T]\ni t\rightarrow D_t\subset \R^N,$$
for Hele Shaw problem which is unique and if $u_t$ be the function appearing in (\ref{weak solution}), then $u_t$ is also unique and $$u_t=\int_{0}^{t}p_{D_{\tau}}\,d\tau.$$
Moreover, $D_t$ can be chosen to be $$D_t=D_0\cup \{z:u_t(z)>0\}.$$
\end{theorem}
For more information about Hele shaw see \cite{gus1}, \cite{r}, \cite{gv} and references therein. 

\section{First numerical  method to approximate the solution of Problem (P).}
\renewcommand{\theequation}{4.\arabic{equation}}
\setcounter{equation}{0}
In this part by applying the properties of free boundary problem, we construct an algorithm that leads us to a fast iterative solver. The level set method
is next employed to evolve the interface in the direction of the normal velocity field.

Consider Problem (\ref{one phase for sub1}) in dimension one.  Our motivation for the first  method is based on the fact that for any $x$ outside of the $\mbox{supp}(\mu)$
one has
\[
u'(x)=\pm \sqrt{2u}.
\]
 To be more precise,  one has
 \be\label{lap u 1}
 \Delta u= 1,  \quad \text{in} \,\,{\{x:u(x) >0}\} \setminus  supp(\mu).
 \ee
  Let $x_f$ be a free boundary point. Multiply (\ref{lap u 1}) by $u'$
 and integrate over $[x \, ,\, x_f]$ to find that
\[
\frac{1}{2}(u')^{2}(x)=u(x).
\]
Let  $(c,d)$ be an  initial guess for ${\{x: u(x)>0}\}$ which contains the support of measure
$\mu$. Next  we solve the following boundary value problem
\be\label{ff1}
 \begin{cases}  u''=1 -\mu &  \mbox{in}  \quad (c, d), \\
u'(c)=\sqrt{2u(c)} , & u'(d)=-\sqrt{2u(d)}.
 \end{cases}
\ee
Then to get the free boundary points, we move the  points  $c, d$  in  the normal direction with speeds  $\sqrt{2u(c)}$ and  $\sqrt{2u(d)}$, i.e,
\[
d_{f}=d+\sqrt{2u(d)}, \quad   c_{f}=c-\sqrt{2u(c)},
\]
 where $c_f$ and $d_f$ are free boundary points. Note that in this case  we   need only one iteration, see section 4.2.
 \begin{remark}
 To have existence for the boundary value problem  (\ref{ff1}) one has  to choose $(c,d)$ close enough  to $\text{supp}(\mu)$.
 \end{remark}

\subsection{Blow up techniques and the main idea.}

In higher dimensions we shall prove  that when we approach to the free boundary still the quotient $\frac{|\nabla u(x)|}{\sqrt{2u(x)}}$ goes to one. First,    we recall some known  properties and lemmas that have been proved  in \cite{psu}. We shall use them in the proof of  Theorem \ref{ff2}. The following lemma shows  the growth of $u$  away from the free boundary $\Gamma$.
\begin{lemma} \cite{psu}\label{Quadratic growth} Let $u \in  L^{\infty}_{\text{loc}}(\Omega)$,~$\Omega ={\{u >
0}\}$    be a solution of Problem (P). If  $ \, x_{0} \in \Gamma$  then
\[
\underset{ B_{r}(x_{0})}{\sup} u \leq  C  r^{2},
\]
where $C=C(N)$.
\end{lemma}
\begin{corollary}  Let $u$  be as in Lemma \ref{Quadratic growth}.  Then
\[
 u(x)  \leq  C  dist(x, \partial \Omega)^{2}.
\]
\end{corollary}
Also we  need the following   \textbf{Non degeneracy} property of the solutions.
 \begin{lemma} \label{non degeneracy} \cite{psu}
  Let $u$  be a solution of given free boundary problem, then we have the inequality
\[
\underset{ \partial B_{r}(x_{0})}{\sup}  \, u \geq  \frac{r^{2}}{8N}, \quad  \, \, \, \, \, \, \text{for any } x_0 \in \Gamma.
\]
\end{lemma}
\begin{definition}
 (Local solutions) For given $R,M > 0,$  and $ x_0 \in  \Gamma,$  let $ P_{R}(x_0,M)$
be the class of $C^{1,1}$ solutions $u$ of Problem (P)  in $B_{R}(x_0)$ such that
\[
|Du(x)-Du(y)|\leq M|x-y|,\qquad \text{for any } x,y\in\R^N.
\]
In the case $x_0 = 0$  we also set $P_{R}(M) = P_{R}(0,M).$
\end{definition}

In the above definition if $R = \infty$ then we get solutions in the entire space $\R^{N}$ and grow quadratically at infinity, which are called  \textit{global solutions}.

If $u \in  P_R(x_0,M)$   and $ \lambda  > 0,$  then the  proper re-scaling of $u$  at $x_0$ is defined by
\[
u_{x_{0}, \lambda}(x)=\frac{u(x_{0}+ \lambda x)-u(x_{0})}{\lambda^2}.
\]
Note that by using non degeneracy, Lemma \ref{non degeneracy}, and quadratic growth properties, Lemma \ref{Quadratic growth},    it can be shown that when
$\lambda \rightarrow 0$  then
$$u_{x_0, \lambda}\rightarrow u_{0}\quad\mbox{in}\,\, C_{\text{loc}}^{1,\alpha}(\R^N)\,\, \mbox{for any}\,\,0<\alpha<1,$$
where $u_0\in C_{\text{loc}}^{1,1}(\R^N)$.
 This $u_{0}$ is called a \textit{blow up} of $u$  with fixed center $x_0$  and also $u_{0}$  is a global solution, i.e,  $ u_{0} \in P_{\infty}(M).$ For  more details see  \cite{psu}.

\begin{theorem}\label{ff3}\cite{psu}
 (Blow up with fixed center). Let $u \in  P_R(x_0,M)$ be a solution of  Problem (P). Suppose that
\[
u_0(x) = \lim _{j\rightarrow \infty }  u_{x_{0},\lambda_{j}} (x), \quad x \in \R^N,
\]
for some sequence $\lambda_{j}\rightarrow  0$   as $j \rightarrow \infty .$ Then $u_0$ is homogeneous of degree two with
respect to the origin, i.e.
\[
u_{0}(\lambda x) =\lambda^{2} u_0(x),  \quad \text {  for    any  } x \in  \R^{N} \text { and } \lambda  > 0.
\]

\end{theorem}

In the proof of next theorem we will use the concept of \textit{regular points}.  A $x_0\in \Gamma$ is a regular point if  every blow up of $u$ at $x_0$ is a half
plane solution. Precisely, there is two category of blowup for a solution of the Problem (P). Let $u_0$ be a blowup with a fixed center
then it has one of the following forms (see \cite{psu}):

\begin{itemize}
 \item Polynomial solution: $ u_0(x) = \frac{1}{ 2} (x \cdot  Ax),  \, x \in   \R^{N}. $ Here $A$  is an $n \times n$
symmetric matrix with $Tr(A) = 1.$
\item Half plane solutions:  $u_0(x) = \frac{1}{ 2}  (x \cdot  e)^{2}_{ +}, \,  x \in  \R^{N} $   where $e$  is a unit vector.
\end{itemize}

\begin{theorem}\label{ff2} Let $x_0 $ be a free boundary  point and  $ x  \in  {\{u>0\}} $  then
\[
\limsup_{x\rightarrow x_0}=\liminf_{x\rightarrow x_0} \frac{ |\nabla u(x)|}{\sqrt{2u(x)}}=1 .
\]
\end{theorem}

\begin{proof}
By Theorem \ref{ff3}, blowup   solutions  at fixed point $x_0 \in \Gamma$ is  a global homogeneous solution of degree two.  Let $u$  be a homogeneous global solution.
Then by above discussion, $u$  has  the following form
\beas
u_0(x) = \frac{1}{ 2}  (x \cdot  e)^{2}_{ +}, \,  x \in  \R^{N} \,\,   \mbox{where}\,\, e  \,\,\mbox{is a unit vector.}
\eeas
Without loss of generality assume that $x_0=0$  then we  know that
\begin{center}$\frac{u(rx)}{r^2}  \rightarrow  \frac{(x_1)_+^2}{2}$  in $ C^{1,\alpha},$ \end{center} which means
\[
{\left| \frac{u(rx)}{r^2}  - \frac{(x_1)_+^2}{2}\right|} \rightarrow 0,
\]
and consequently,
\[
{\left|\frac{\nabla u(rx)}{r} -x_{1} e_{1}\right|}\rightarrow  0.
\]
From above one can get
\[
u(rx)=\frac{(rx_{1})_+^2}{2}+ c r^2 , \text{   where } c  \text{  is  an arbitrary small constant and}
\]
\[
\nabla u(rx)=(rx_1)e_1+O(r^\alpha),\quad  \alpha<1.
\]
Using above expression for  $u(rx)$  and  $|\nabla u(rx)|$  and taking the quotient implies the limit.
\end{proof}


  \subsection{A mixed boundary value problem and first algorithm.}
 Assume  $(\Omega, u)$ be a smooth solution of Problem (P). Our aim is to build a sequence $(\Omega_k, u_k)$ of solutions of an
approximate quadrature domain problem which converges towards $(\Omega, u).$ Assume that $p_k \in \pa \Omega_k$. Let $\mathbf{n}_k$ be the normal outward vector on $\pa \Omega_k$. By Taylor formula, one can write
 \[
 u(p_{k} +d_k \mathbf{n}_{k})\simeq  u(p_{k})  + d_k \nabla u(p_k)\cdot \mathbf{n}_{k}+ \frac{d_{k}^{2}}{2}\, \mathbf{n}_{k}^{T}\cdot D^{2}u(p_k)\cdot \mathbf{n}_k.
 \]
 We wish  to have  $u(p_{k} + \mathbf{n}_{k}d_k)=0$, so   if we put  $\nabla u(p_k)\cdot \mathbf{n}_{k}=-\sqrt{2u(p_k)}$ and use the approximation $D^{2}u\simeq \frac{1}{2} (\Delta u) I $, then one gets
 \be\label{dk}
 d_k=\zeta\sqrt{2u(p_k)},
  \ee
  where $\zeta= 2-\sqrt{2}$. It means that if $\Gamma_k=\pa \O_k$ then $\{\Gamma _k+d_k\cdot \mathbf{n}_k\}$ converges to $\Gamma$. We note that in dimension one,  $d_k=\sqrt{2u(p_k)}$.


To construct an algorithm, let $\O_0$ be an initial guess of $\O$ which contains $supp(\mu)$. Consider the following  boundary value problem which has a vital role
in the numerical scheme
\be\label{algo}
\begin{cases}
\Delta u= 1 -\mu, &  \mbox{ in } \O_0,\\
\frac{\partial u}{\partial n}=-\sqrt{2u}, & \mbox{  on  }   \partial \O_0.
\end{cases}
  \ee

\begin{remark}
 We note that  (\ref{algo}) is not stable at the points close to the free boundary,  therefore alternatively  we solve the following problem to have more  efficient and robust scheme,
\be\label{equation in algorithm}
\begin{cases}
\Delta u_{k} = 1 -\mu, &   \mbox{ in } \Omega_{k},\\
\frac{\partial u_{k}}{\partial n_k}=-\theta u_k, & \mbox{  on  }   \partial \Omega_{k}.
\end{cases}
  \ee
We desire that $\theta u_k$ behaves like $\sqrt{2u_{k}}$,  therefore one is able to choose
\be\label{theta}
\theta=\bigg ({\frac{2}{\displaystyle\sup_{\pa \O_{k-1}} u_{k-1}}}\bigg)^{1/2}.
\ee
To determine an optimal value for $\theta$ we employe the fixed point process and iterate $\theta$ as follows.

First of all, let $\theta$ be as in (\ref{theta}) and solve (\ref{equation in algorithm}) to obtain the value of $u$ on $\pa \O_k$. Then compute the corresponding $\theta$ by the formula (\ref{theta}) and repeat the same process again. This scenario will  converge  to the best choice of $\theta$.
\end{remark}
The existence of (\ref{equation in algorithm}) is
based on minimization techniques and is a special case of the next lemma.

\begin{lemma}\cite {ev}\label{exercise in Evans}
 Assume $\beta\, :\,\R \rightarrow \R $ is smooth, with
$$0<a\leq \beta '(z)\leq b\quad (z\in \R),$$
for constants $a,b$. Let $f\in L ^2(U)$, $U\subset \R^N$ is a bounded, open set with smooth boundary. For
\beas
\begin{cases}
-\Delta u= f, &  \mbox{in}\,\, U,\\
\frac{\pa u}{\pa n}+\beta(u)=0, & \mbox{on}\,\, \pa U,
\end{cases}
\eeas
there exists a unique weak solution.
\end{lemma}

\subsection{Level set formulation.}
The level set method was introduced by Osher and Sethian for implicitly tracking dynamic surfaces and curves, see \cite{of,se}. The main idea behind this method is to embed an interface $\Gamma$, which lies in $\R^{N-1}$ into a surface in dimension $\R^N$. We can do this embedding by defining  a proper function $\phi$ such that $\Gamma $ is the zero level set of $\phi$, i.e,
$$\Gamma=\pa \O=\{x\in \R^N;\,\phi(x)=0\}.$$
Suppose that $\Gamma$ divides $\R^N$ into multiple connected components then one can recognize the inside of one component from its exterior when the sign of $\phi$ changes.

Regarding to Theorem \ref{estimate}, let $\mathcal{T}$ be a given rectangle such that $\O\subset B_{r(\mu)+R}\subset \mathcal{T}$ for appropriate $R>0$. To apply the level set method for Problem (P), we need $\phi$ be  positive in $\mathcal{T}\setminus \O$ and negative in $\O$. By this way the outward normal vector of $\O$ is given by
$$\bf n=\frac{\nabla \phi}{|\nabla \phi|}.$$

  We note that  Problem (P) is stationary and the level set formulation requires a time evolution so we define the parameter $t$ and
  introduce  a family of boundaries $\O(t)$ for $t>0$ as the level sets by
  $$\pa \O(t)=\{x\in \R^N;\,\phi(x,t)=0\},$$
  for unknown function $\phi\,:\, \mathcal{T}\times \R^+\rightarrow \R$.  By chain rule
  \[\phi_t+\nabla \phi(x(t),t)\cdot x'(t)=0.\]
  Let $F=x'(t)\cdot \bf n$ which means that $F$ is  speed   in outward  normal direction. Then the level set equation will be as follows
  \[
  \begin{cases}
  \phi_t+F|\nabla \phi|=0,\\
  \phi(x,t=0)\,\,\text{is given}.
  \end{cases}
  \]

In this paper we restrict our attention to the case that $\phi$ is considered as the sign distance function and therefore $|\nabla \phi|=1$. Hence the level set equation  turns to
 \be\label{dis}
 \frac{\pa \phi}{\pa t}+ F = 0 \quad \text{in} \quad   \mathcal{T}\times \R^+.
 \ee

Now  consider the following boundary value problem
 \be
 \begin{cases}
  \Delta u(t) =1-\mu, & \mbox{in}  \quad \Omega(t), \\
\frac{\partial  u(t)}{\partial n} =-\theta u(t),&  \mbox{on}  \quad \partial \Omega(t).
 \end{cases}
\ee
According to  (\ref{dk}), we choose the quantity $\zeta\sqrt{2u(t)}$ as the speed which decreases in $\O(t)\setminus supp(\mu) $ and  goes to zero when  $\O(t)$  approaches to the free boundary.  Regarding to (\ref{dis}), the displacement of the boundary $\O(t)$ can be obtained  by considering the following equation :
\be\label{extention1}
\frac{\pa \phi}{\pa t}+\zeta\sqrt{2u(t)}=0,\quad\mbox{on}\,\,\pa \O(t).
\ee
 Now let $\mathcal{T}$ be the rectangle  in section 4.2. The extension of the previous equation to whole domain $\mathcal{T}$,  is one of the important issue in the level set approach. To do this we solve the problem:
\be\label{ext}
\begin{cases}
   \Delta v(t) = 1, & \mbox{in}\quad  \O(t)\setminus supp(\mu), \\
   \Delta v(t) = 0, & \mbox{in}\quad  \mathcal{T}\setminus \O(t),\\
 v(t)  = 0, &\mbox{on} \quad \pa (supp(\mu))\cup \pa \mathcal{T}, \\
 v(t) = \zeta\sqrt{2u(t)},&  \mbox{on}  \quad \partial \Omega(t).
\end{cases}
\ee
We now extend  equation (\ref{extention1}) to $\mathcal{T}$ by
\be\label{extention2}
\frac{\pa\phi}{\pa t}+v(t)=0,\quad\mbox{in}\,\,\mathcal{T}\setminus supp(\mu).
\ee
For more information on velocity extension see \cite{fr,of}.


\subsubsection{First algorithm for Problem (P).}

 Choose a  tolerance, TOL$<<1$.

\begin{enumerate}
\item Set $ k=0 $, choose an initial domain $ \Omega_0  $  with   $ \Gamma_{0}=\partial \Omega_0$ such that
\[
supp(\mu) \subset \Omega_{0} \subset B_{r(\mu)+R}.
\]

\item Compute $u_k$ on $\Omega_k$ which is the solution of the following elliptic boundary value problem
\beas
(\star)\qquad
\begin{cases}
\Delta u_{k} = 1 -\mu, &  \mbox{in } \Omega_{k},\\
\frac{\partial u_{k}}{\partial n_k}=-\theta u_k, & \mbox{on }   \partial \Omega_{k}.
\end{cases}
  \eeas

\item Solve (\ref{ext}) and obtain $v$.
\item Update the level set function $\phi$ from (\ref{extention2}) to get $\O_{k+1}$.

\item Solve $(\star)$ in $\O_{k+1}$ and get $u_{k+1}$.

\item If $\displaystyle\sup_{\pa \O_{k+1}}|u_{k+1}|< \mbox{TOL}$, then stop else set $k=k+1$ and go to (2).

\end{enumerate}

\section{Second  numerical method to approach to the solution of Problem (P) based on shape optimization.}
\renewcommand{\theequation}{5.\arabic{equation}}
\setcounter{equation}{0}
 The shape sensitivity analysis is used to define a velocity field, which allows us to update the surface while decreasing
a given cost function.
The solution of an elliptic boundary value problem usually depends highly
nonlinearly on the geometry of the given domain. Thus the geometry can not be solved
straightforward from a linear equation.

In shape optimization approach,  we rewrite the free boundary problem
 such that the minimum of some cost functional  is attained at the solution of free boundary.
The solution of Problem (P)   minimizes the  functional
\be\label{boundary2}
E(u,\Omega)=\int_{\Omega}\frac{1}{2}|\nabla u|^2 dx + \int_{\Omega}(1-\mu)\, u\,dx,
\ee
over $u\in H^{1}(\Omega)$ where $\Omega={\{u>0}\}.$  Note that we get $u=0$ on $\partial \Omega.$


  In the following we discuss the shape sensitivity analysis for the above shape functional  related to  Problem (P). At first,   we briefly recall some basic   facts related to shape calculus \cite{sz}.

 In shape sensitivity  we  analyze how the solution of a
  PDE changes when the domain is changing with a velocity field. Let $ x \in \R^N  $, and  $ \mathbf{V }(t, x) $ be a  velocity field (vector field)  defined in
   $D,  \mathbf{V} \in C^{k}(D;\R^N), \mathbf{V}|_{\partial D}=0 $.
Let $ t $  be  artificial time.
Assume that  $ \Sigma \subseteq D$. It is natural to define transformation $T_{t}(\mathbf{V} )x =X(t, x)$ with a velocity field $\mathbf{V}$ by  differential equations
 \[ \frac{\partial X}{\partial t}(t,x) =\mathbf{V}(t,x), \ \ X(0, x) = x, ~~~~  x \in \Sigma.
 \]
One can see that this   transformation is quite close to a perturbation of the identity in \cite{sz,de}, where
the transformation was defined  by
\[
T_{t}(\mathbf{V} ) = I + t\mathbf{V} (x).
\]
For small perturbations these two transformations are close (see \cite{ti2}). The image of $\Sigma\subset \O$ under $T_t$ is $\Sigma_t$.

 Let $J$ be a domain functional $ J:\Sigma\longmapsto  \mathbb{R}$ .
We say that the functional has a directional shape derivative to direction $\mathbf{V}$ at $\Sigma $   if the limit
\[
 \underset{t \rightarrow 0}    {\text {lim}}  \frac{J(\Sigma _{t})-J(\Sigma)}{t}:=d J(\Sigma,\mathbf{V}), \\
\]
exists. If further $dJ (\Sigma, \mathbf{V} )$ is linear and continuous with respect to $\mathbf{V}$ and it exists for all directions $\mathbf{V}$, we say that
$J$ is shape differentiable at $\Sigma$.
 By Hadamard's  structure theorem,  $ dJ(\Sigma,\mathbf{V}) $ depends  only on the normal
component  of $\mathbf{V} $ on the boundary of $\Sigma $, see \cite{ti,zo}.

We use the notations $u_{\Omega}$ or $u(\Omega) $ to show the dependence of solution of a given PDE with respect to  the domain $\O$. For a function
$v(\Sigma)$ and $\Sigma\in C^k,\,k\geq 1$,   we define material derivative  as a limit
 \[
\dot{v}(\Sigma;\mathbf{V})(x) := \underset{t \rightarrow 0}  { \text {lim}}  \frac {v(\Sigma_t)\circ T_t(\mathbf{V})-v(\Sigma)}{t}.
\]
This limit may exist  either in a weak or a strong sense, and the material derivative is
called  a weak or strong material derivative respectively, see \cite{sz}.

The  shape derivative of $v(\Sigma)$ in the direction $\mathbf{V}$ is the element $v'(\Sigma; \mathbf{V})$  defined by
\[
v'(\Sigma;\mathbf{V}) :=\dot{v}(\Sigma;\mathbf{V})-\nabla v(\Sigma)\cdot \mathbf{V}(0),
\]
whenever it exists either in a weak or a strong sense. For simplicity's sake we shall utilize  $v'_\S$ instead of $v'(\Sigma;\mathbf{V}).$

 Shape derivative  represents the change of function $v$ with respect to the geometry. Equivalently, shape derivative is the variation of the state variable with respect to the shape change.

The following lemmas represent the basic formulas for shape differentiation of integrals. In the following we assume that $\O$ is bounded.

\begin{lemma}\cite{sz} \label{for derivative}
 Let $ f (\O_t)\in L^{1}(\Omega_t) $ be shape differentiable and $f'(\O_t)\in L^1(\O_t)$, $t\in[0,T]$ and $T>0.$ If $\O_t$ is a $C^{0,1}$-domain, then
\begin{equation}\label{derivat}
\bigg(\frac{d}{dt}\int_{\Omega_{t}} f(\O_t) dx \bigg)\bigg|_{t=0}=\int_{\Omega}f'(\O)dx + \int_{\partial\Omega} f(\O)<\mathbf{V},\mathbf{n}> ds.
\end{equation}
\end{lemma}


\subsection{Shape optimization techniques for Problem (P) and second algorithm.}

First ingredient  is the shape derivative of the function $u_\Omega$.

\begin{lemma}\label{u prim}
The shape derivative of $u_\Omega$ in the normal direction $\mathbf{V}$, is given by the function $u'_\Omega$, satisfies
\be\begin{cases}\label{derivative}
\Delta u'_\Omega =  0, & \text{in}\quad\Omega, \\
u'_\O = -\frac{\pa u }{\pa \mathbf{n}} <\mathbf{V}(0),\mathbf{n}>,& \mbox{on}\quad \pa \O.
\end{cases}\ee
\end{lemma}

\begin{proof}
The minimizer of the functional in (\ref{boundary2})  satisfies the following equation
\[
\Delta u_{\O_t}=f =1-\mu  \quad   \text {in } \O_t. \\
\]
By multiplying a test function, $ \varphi \in H^{1}_{0}(\Omega_t),$ and taking integral one obtains
\begin{equation}\label{shape2}
\int_{\O_t} \nabla u_{\O_t}  \cdot \nabla \varphi  \ dx= -\int_{\O_t} f \, \varphi \, dx.
\end{equation}\\
Taking the derivative of the above equation respect to $t$ and considering  Lemma \ref{for derivative} one can see that $ u'_{\O} $ satisfies
\[
\int_{\O} \nabla u'_\O \cdot \nabla \varphi  \, dx= -\int_{\pa \O} f'\,dx=0.
\]
That is
\[
\Delta u'_\O=0.
\]
The boundary condition in   (\ref{derivative}) is verified by equation (3.6), chapter 3 in \cite{sz}.
\end{proof}

\begin{remark}
Let $\Gamma=\partial  \Omega $ be the free boundary for the solution of Problem (P). Then
\beas
u'_{ \Omega} = 0 \qquad \text{in } \O.
\eeas
\end{remark}

Let us now to analyze the behavior of the energy near the solution.
\begin{lemma}
Consider the energy functional (\ref{boundary2}) of  Problem (P). Then the shape derivative of $E$ with respect to $\mathbf{V}$ is
\be
d E(\S,\mathbf{V})\,=\,\int_{ \S} div(-\frac{1}{2}|\nabla u|^2\,\mathbf{V})\,dx.
\ee
\end{lemma}

\begin{proof}
By Lemma \ref{for derivative} one can see
\beas
d E(\S,\mathbf{V})=\int_{\S}\big(\nabla u  \cdot \nabla  u' +(1-\mu)u'\big)\,dx +\int_{\pa \S} \big( \frac{1}{2} |\nabla u|^2+(1-\mu)u\big)\mathbf{V}\cdot \mathbf{n}\, ds, \eeas
where $u'$ is the shape derivative of $u$ into direction $\mathbf{V}$. Our assumption on Problem (P) states that $u\lfloor_{\pa \S}=0$. Then the shape derivative of $E$ is
\be\label{derivative of E}
d E(\S,\mathbf{V})=\int_{\S}\nabla u  \cdot \nabla  u' \ dx+\int_{\S}(1-\mu)u'+ \int_{\pa \S}  \frac{1}{2} |\nabla u|^2\,\mathbf{V}\cdot\mathbf{n}\, ds.
 \ee
According to Green's theorem, the first term of (\ref{derivative of E}) is
\beas\label{derivative of E2}
\int_{\S}\nabla u  \cdot \nabla  u' \,dx=-\int_{\S}u'\,\Delta u\,dx+ \int_{\pa \S} u'\frac{\pa u}{\pa \mathbf{n}}\,ds,
 \eeas
and we get
\begin{align}
 \notag dE(\S,\mathbf{V})=&-\int_{\S} u'\Delta u\,dx+\int_{\pa \S}u'\frac{\pa u}{\pa \mathbf{n}}ds\\
\notag &+\int_{\S} (1-\mu) u'\,dx+ \int_{\pa \S}\frac{1}{2} |\nabla u|^2\,\mathbf{V}\cdot \mathbf{n}\, ds\\
\notag &=-\int_{\S} u'(1-\mu)\,dx+\int_{\pa \S}u'\frac{\pa u}{\pa \mathbf{n}}ds\\
\notag &+\int_{\S}(1-\mu) u'\,dx+ \int_{\pa \S}\frac{1}{2} |\nabla u|^2\,\mathbf{V}\cdot \mathbf{n}\, ds\\
\notag &= \int_{\pa \S}u'\frac{\pa u}{\pa \mathbf{n}}ds+ \int_{\pa \S}\frac{1}{2} |\nabla u|^2\,\mathbf{V}\cdot \mathbf{n}\, ds.
\end{align}

As $u$ is the solution of a Dirichlet problem, Lemma \ref{u prim} gives us $u'=-\frac{\pa u}{\pa \mathbf{n}}<\mathbf{V},\mathbf{n}>$ on $\pa \S$. Hence we have for $dE(\S,\mathbf{V})$  the expression
\beas
d E(\S,\mathbf{V})=- \int_{\pa \S}\frac{1}{2} |\nabla u|^2\,\mathbf{V}\cdot \mathbf{n}\, ds,
\eeas
and by Stock's theorem it turns
\beas
d E(\S,\mathbf{V})\,=\,\int_{ \S} div(-\frac{1}{2}|\nabla u|^2\,\mathbf{V})\,dx.
\eeas
\end{proof}

\begin{corollary}

The solution of Problem (P)  is a critical point of the energy functional  $E$.

\end{corollary}

\begin{proof}
We choose  $\mathbf{V}\cdot \mathbf{n} = - \frac {\partial u_{\Sigma}}{\partial \mathbf{n}}  $  on $ \partial \Sigma $. If $ \Sigma  \subset \Omega$ then $\frac {\partial u_{\Sigma}}{\partial \mathbf{n}}<0$ so we have   $d E(\S,\mathbf{V})\leq 0$ and it means that $E$ is decreasing respect to $V$ and the
 solution of free boundary where $\nabla u=0$, is a critical point of $E$.
\end{proof}


\subsubsection{Second algorithm for  Problem (P).}

\begin{enumerate}
\item Set $ k=0 $, choose an initial domain $ \Sigma_0  $ such that $supp(\mu)\subset \S_0$ and set  $ \Gamma_{0}=\partial \Sigma_0 $.

\item Solve $\Delta u_{k} =  1$ in $\S_k\setminus supp(\mu)$ with Dirichlet  boundary condition $ u_{k} = 0 $  on $ \Gamma_k $,

\item Compute a normal velocity from (2), i.e.
\beas
\mathbf{V}\cdot \mathbf{n} = -  \nabla u_k \cdot \mathbf{n}_{\Gamma_k}.
\eeas

\item Stop if $\Vert \nabla u_k \Vert_{L^2(\Gamma)}$ is sufficiently small.

\item Given  $\Gamma_{k}, $ move the free boundary by Quasi-Newton method, i.e,

In dimension one
\beas
x_{k+1}=x_k-{u'(x_k)}.
\eeas

In dimension two
\beas
\Gamma_{k+1}=\Gamma_{k}-{\nabla u({x}_k)}\cdot I.
\eeas
Obtain the new shape $\Sigma_{k+1}$ with the free boundary $\Gamma_{k+1}$.

\item Set $k=k+1$ and go to (2).

\end{enumerate}

\subsection{Alternative viewpoint.}

One can consider another starting point.  We try to determine a shape $\Omega$ such that
\beas
\frac{\partial u_\Omega}{\partial \mathbf{n}} = 0, \qquad \text{on } \Gamma.
\eeas
In order to derive a suitable weak formulation, we multiply the normal derivative by a smooth
test function $\varphi$ and integrate over $\Gamma$, i.e. we have
$$ \int_\Gamma \frac{\partial u_\Omega}{\partial \mathbf{n}} \varphi~d\sigma = 0. $$
By Gauss' Theorem together with the Poisson equation for $u_\Omega$ we have
\beas
\int_\Omega ( f \varphi + \nabla u_\Omega \cdot \nabla \varphi)~dx = 0, \qquad \forall~\varphi \in H_0^1(\Omega).
\label{eequation1}
\eeas
In other words,  the first optimality condition for $E$ (with respect to $v$) reads
\[
dE(u;\varphi,\O):=dE(u+\varepsilon\,\varphi,\O)\lfloor_{\varepsilon=0}=\int_\Omega ( f \varphi + \nabla u_\Omega \cdot \nabla \varphi)~dx = 0,
\]
for all $u\in H_0^1(\Omega)$. If one consider
\beas
J(\varphi,\O) = \int_\Omega ( f \varphi + \nabla u_\Omega \cdot \nabla \varphi)~dx  \label{edefinition},
\eeas
then  $J(\Omega,.)$ is a continuous linear functional on $H_0^1(\Omega)$, i.e,
it can be interpreted as an element of $H^{-1}(\Omega)$ and we can define an operator
$F(\Omega) = J(\Omega,.)$ mapping into $H^{-1}(\Omega)$ such that \eqref{eequation1} is equivalent to solving
\begin{equation}
F(\Omega) = 0 \qquad \text{in } H^{-1}(\Omega).
\end{equation}
Now we can do all similar calculations for the functional $J$ and deduce same results.

\section{Numerical examples.}
\renewcommand{\theequation}{6.\arabic{equation}}
\setcounter{equation}{0}

\begin{example} Suppose that $\mu=(1+2x^2+y^2)\chi_{B}$ where $B$ is the unit ball. We apply the first method and compute the corresponding quadrature domain for Problem (P). Let $B$ be the initial guess. Figure \ref{f1} depicts the numerical solution after just three iterations and Figure \ref{f2} demonstrates the quantity of $|\nabla u|$ on the boundary of solution. This example confirms that the first method is a fast iterative solver.
\end{example}

\begin{example} In this example the  support of the measure $\mu$ is  a polygon which is shown in Figure \ref{first iteration}. We employ the second method and obtain the corresponding quadrature domain.    Set $\mu=1.5 \chi_P.$ In Figure \ref{first iteration},  the initial guess  is the circle and  this figure  shows the solution after first iteration.  Figure \ref{final iteration} states the  result after four iterations. Figure \ref{cross}  illustrates the norm of the gradient on the boundary of  the solution in forth iteration.

 Now let $\mu=11 \chi_P$. Figure \ref{first iteration2} shows the solution after first iteration  and Figure
\ref{final iteration2} states the final result which is close to a ball. Figure \ref{cross2}  illustrates the quantity
of $\frac{|\nabla u|}{\sqrt{2u}}$ on a cross section line which has been shown in Figure \ref{cross and solution}. This Figure verifies Theorem \ref{ff2}.
\end{example}

\begin{example}
Suppose that  $\mu=t(\chi_{B_1}+2\chi_{B_2})$ is uniformly distributed on two circles $B_1(x_1,1),B_2(x_2,1)$ where $x_1=(-2,0), x_2=(\sqrt{8},0)$. According to Example \ref{example} or Remark \ref{rem1}, we find that if $t=4$ then $B_1$ and $B_2$  touch each other at origin tangentially. Consider the second method and let time increase to $t=5$ and  solve
\be\label{exam2}
\begin{cases}
\Delta u= 1-t(\chi_{B_1}+2\chi_{B_2}),& \mbox{in} \,\,\O,\\
u=0,& \mbox{on} \,\,\pa \O,
\end{cases}
\ee
to get the corresponding quadrature domain.
Figure \ref{s4} shows the solution at  $t=5$ and Figure \ref{grad5} illustrates $|\nabla u|$ for $t=5$.  Figure  \ref{s10} is the solution of similar PDE for $t=6$.
\end{example}

\begin{figure}[!h]
  \centering
   \includegraphics[width=0.7\columnwidth] {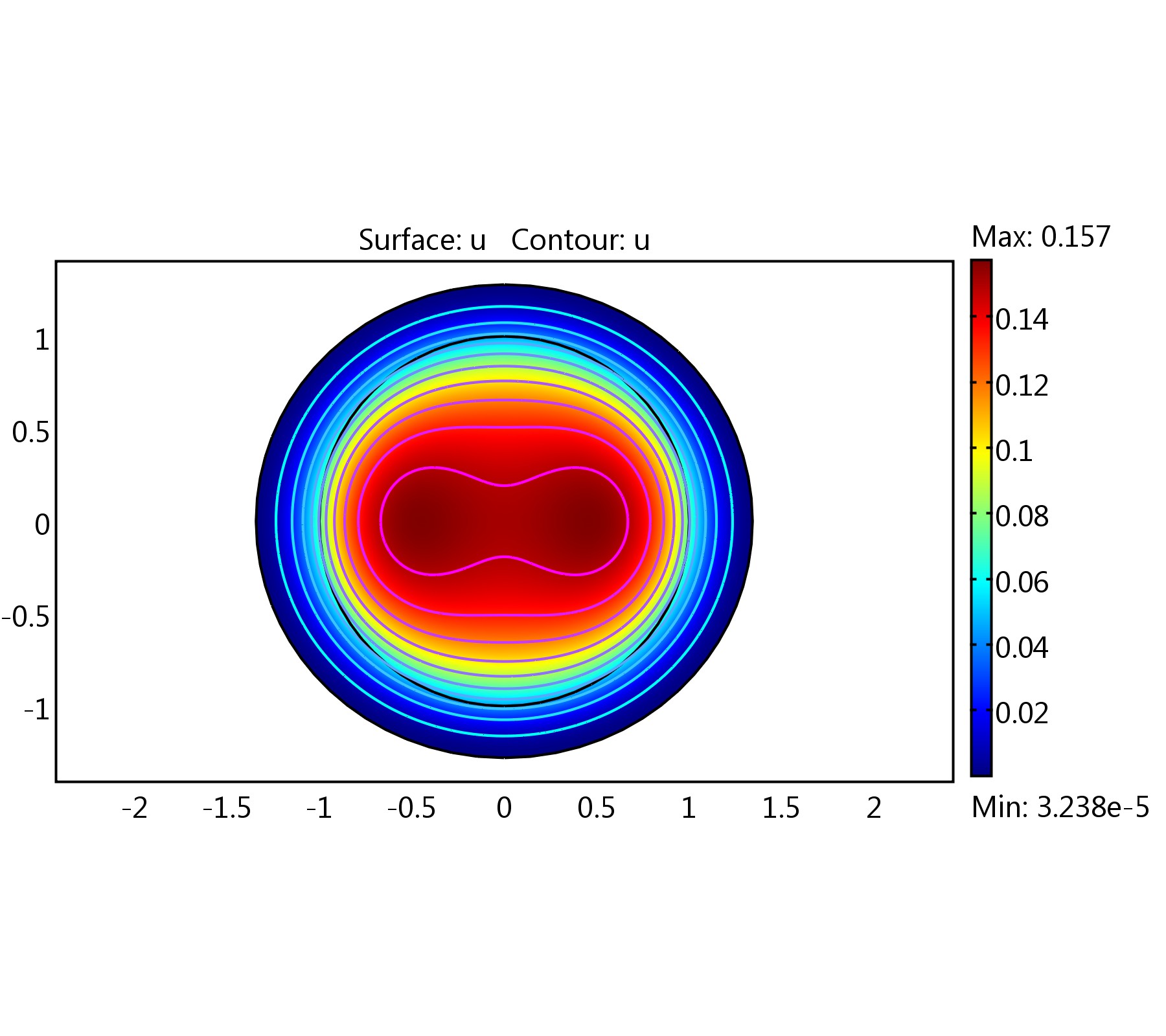}
    \caption{The solution of Problem (P) by employing the first method and considering  $\mu=(1+2x^2+y^2)\chi_{B}$ after three iterations. }
    \label{f1}
\end{figure}

\begin{figure}[!h]
  \centering
   \includegraphics[width=0.7\columnwidth] {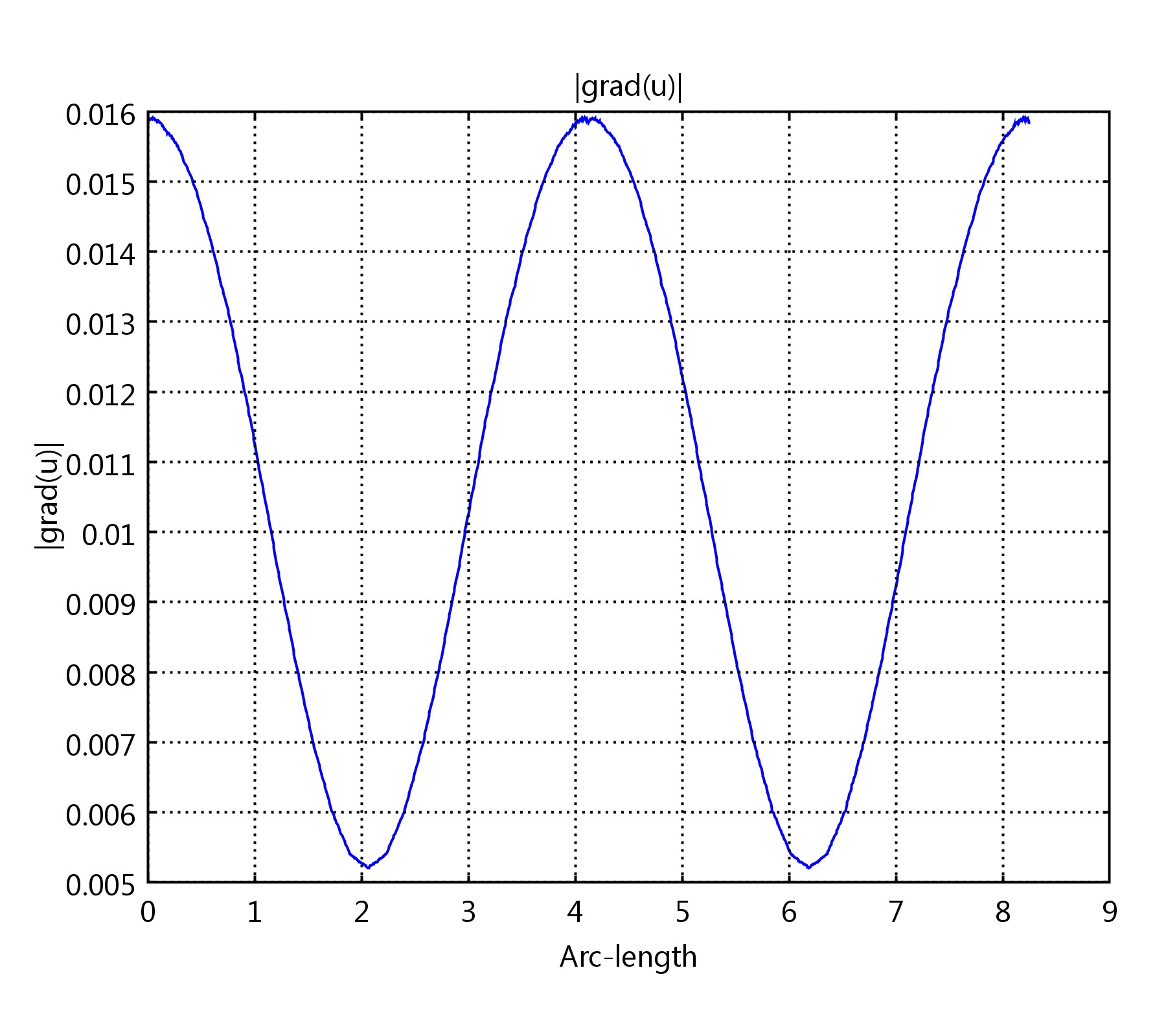}
    \caption{The quantity $|\nabla u|$ on the boundary of the solution in Figure \ref{f1}. }
    \label{f2}
\end{figure}

\begin{figure}[!h]
  \centering
   \includegraphics[width=0.7\columnwidth] {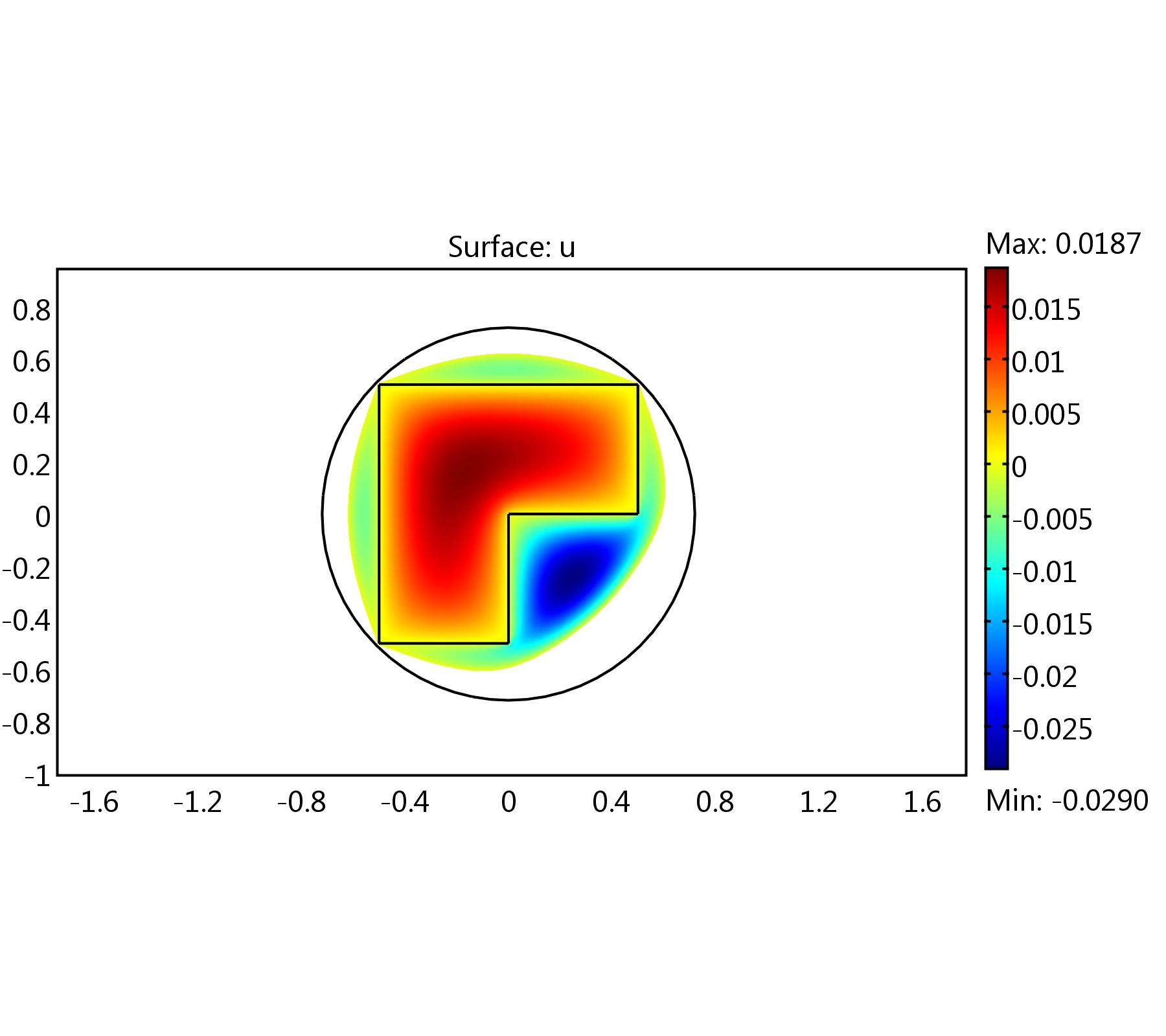}
    \caption{The colored part shows the solution  $\Omega_{1}$, after first  iteration, where support of $\mu$  is the polygon and the initial guess ($\Omega_{0}$) is a ball. }
    \label{first iteration}
\end{figure}

\begin{figure}[!h]
  \centering
   \includegraphics[width=0.7\columnwidth] {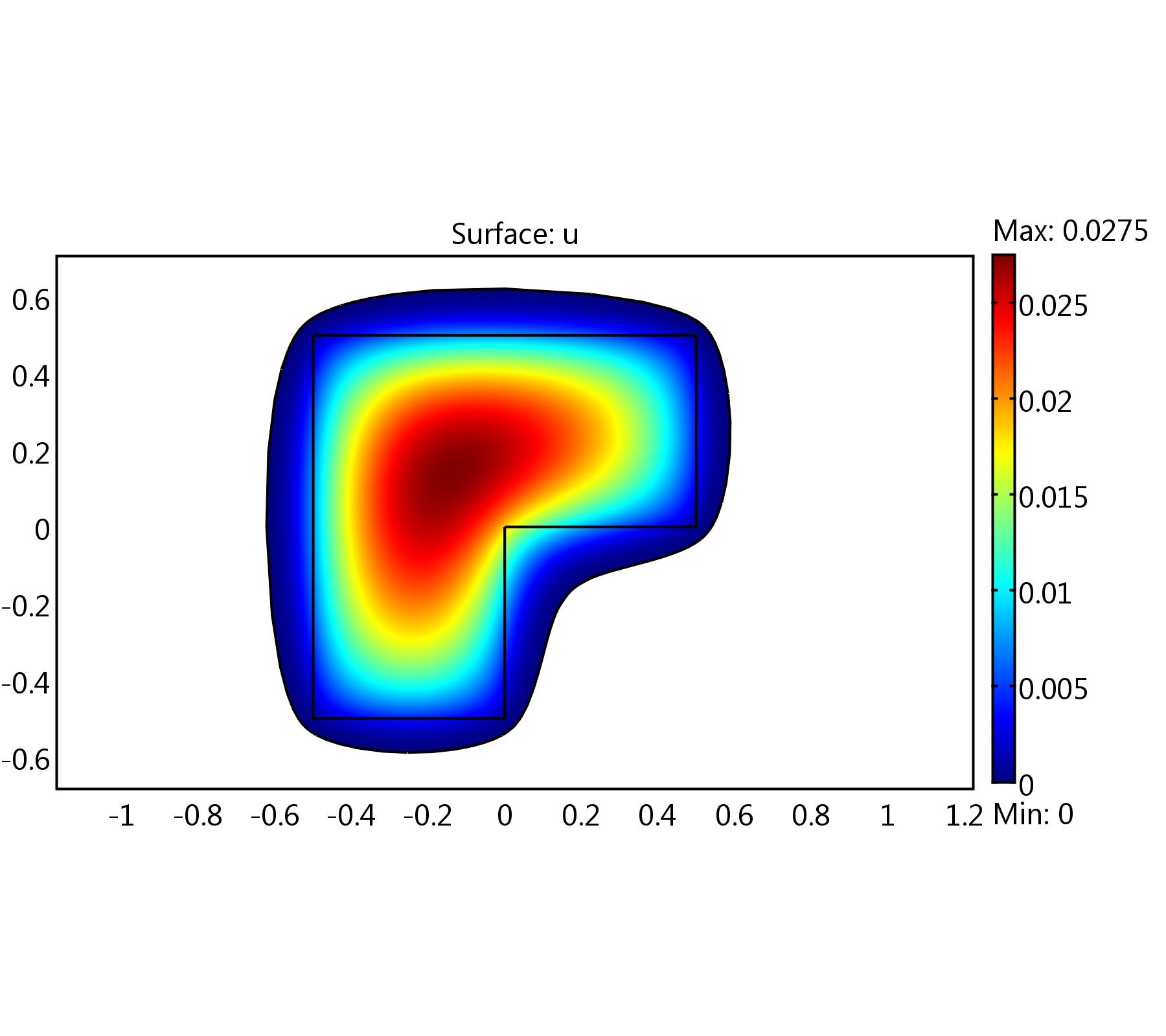}
    \caption{Final domain after four iterations when  $\mu=1.5\chi_P$ and where $P$ is the polygon. }
    \label{final iteration}
\end{figure}


\begin{figure}[!h]
  \centering
   \includegraphics[width=0.7\columnwidth] {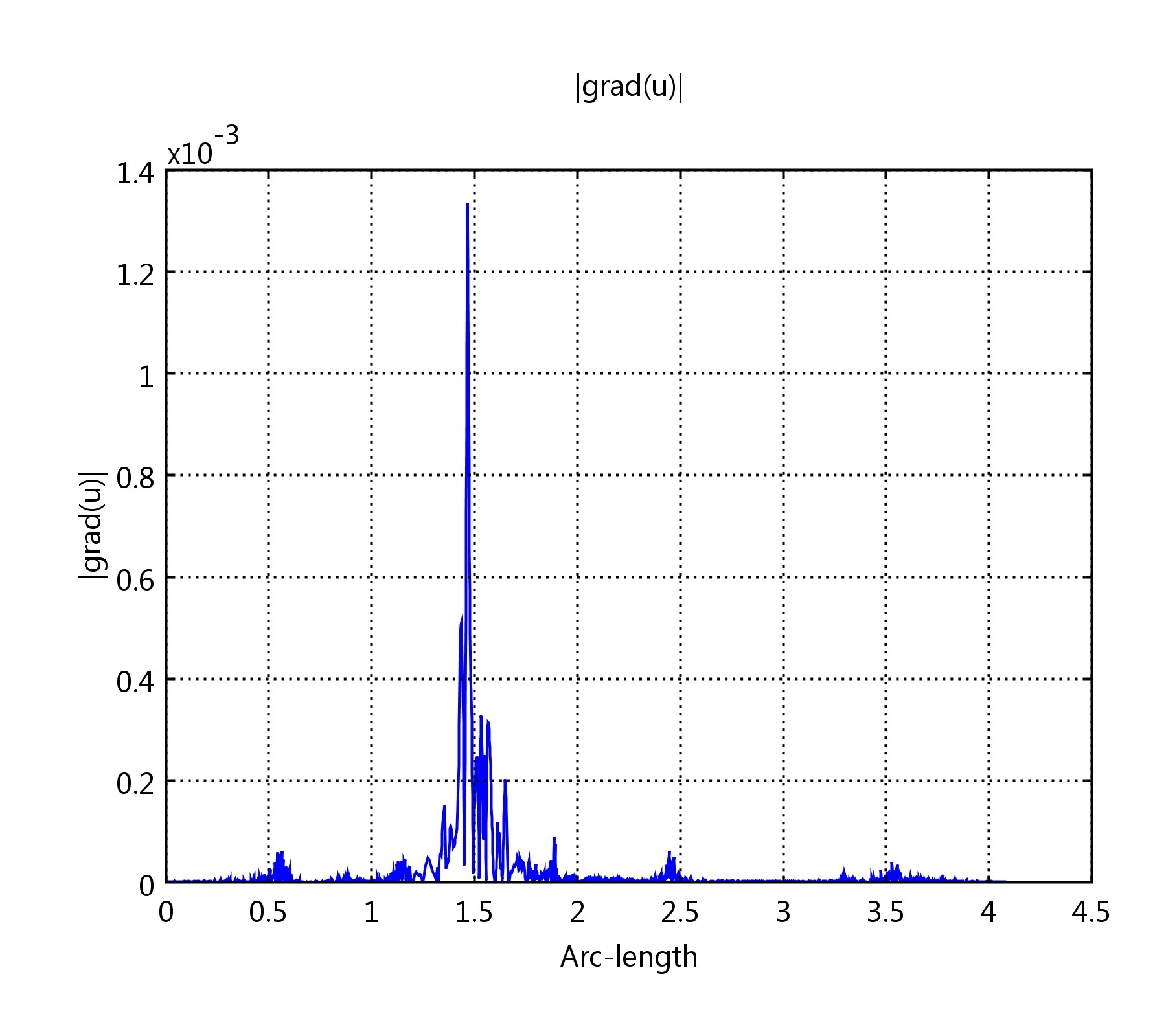}
    \caption{ The value of   $|\nabla u|$ on the boundary of the solution after four iterations.}
    \label{cross}
\end{figure}

\begin{figure}[!h]
  \centering
   \includegraphics[width=0.7\columnwidth] {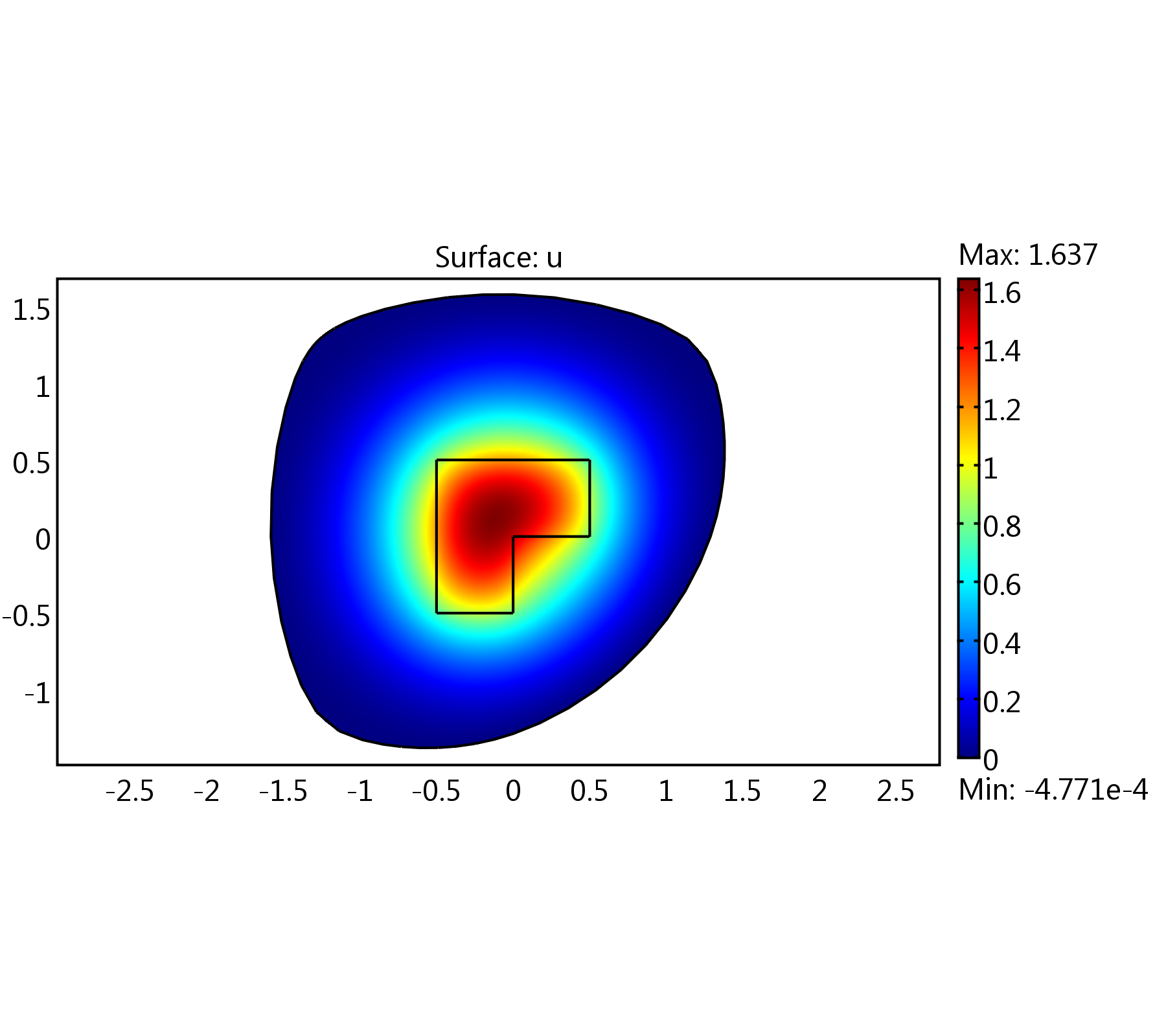}
    \caption{The first iteration for $\mu=11 \chi_P,\,$ where $P$ is the polygon. Initial guess is a ball with center at origin. }
    \label{first iteration2}
\end{figure}

\clearpage

\begin{figure}[!h]
  \centering
   \includegraphics[width=0.7\columnwidth]{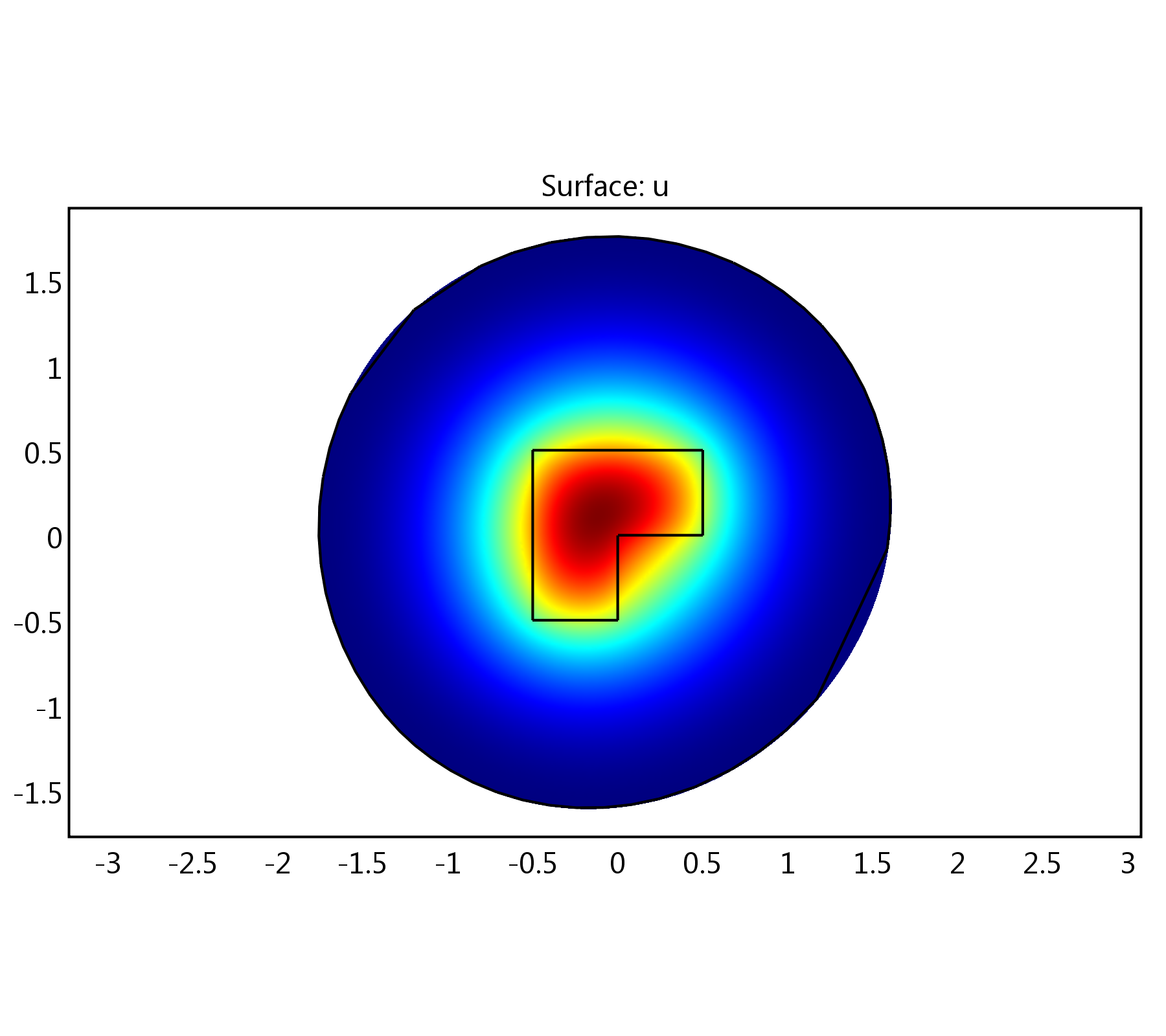}
    \caption{Final quadrature domain when $\mu=11\chi_P$ and  where $P$ is the polygon. }
    \label{final iteration2}
\end{figure}
\begin{figure}[!h]
  \centering
   \includegraphics[width=0.7\columnwidth  ] {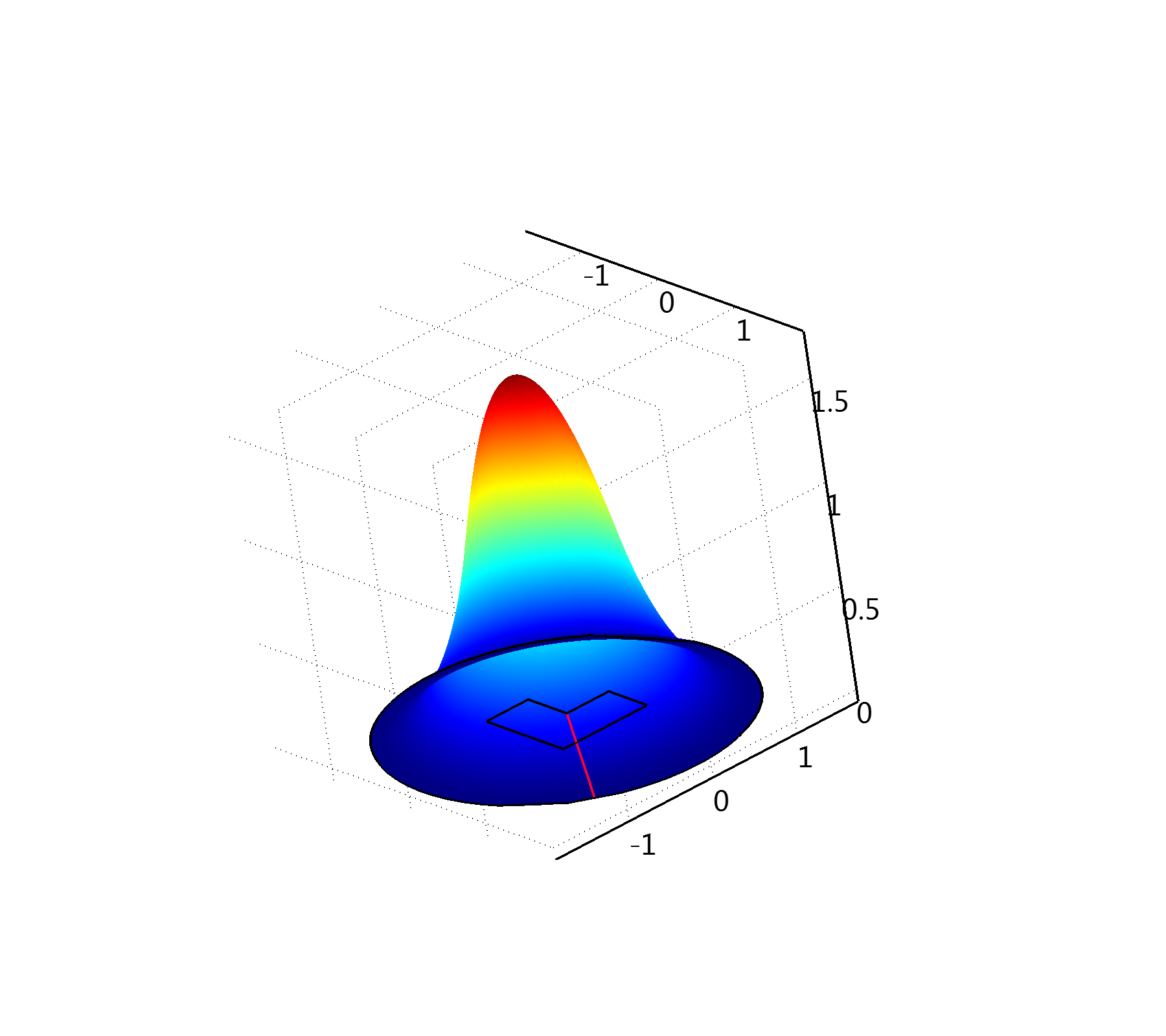}
    \caption{The surface of the solution   $u$  and a  cross section line.}
    \label{cross and solution}

\clearpage

\end{figure}
\begin{figure}[!h]
  \centering
   \includegraphics[width=0.8\columnwidth  ] {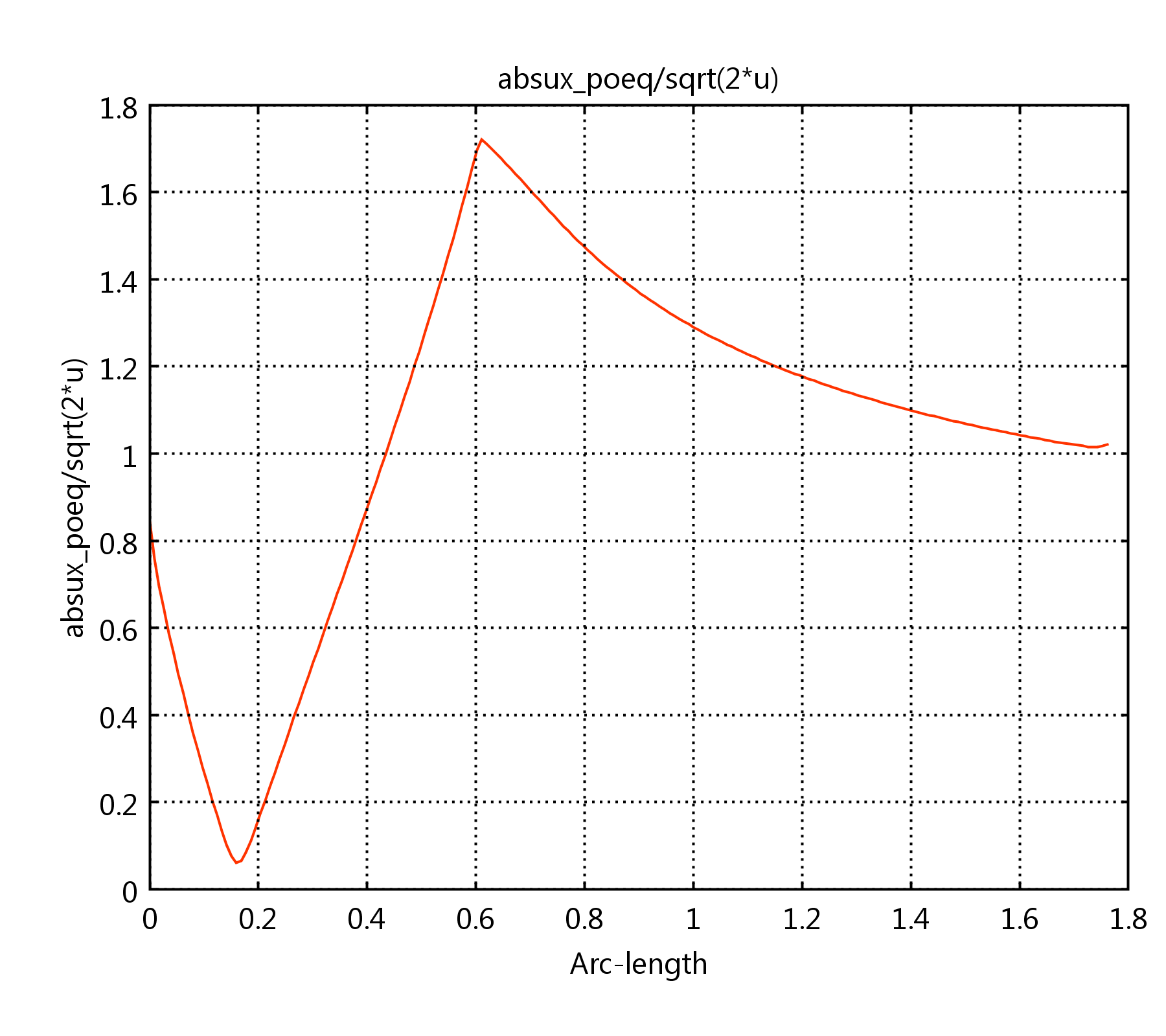}
    \caption{The amount of $\frac{|\nabla u|}{\sqrt{2u}}$ on the cross section line  in  figure (\ref{cross and solution}).}
    \label{cross2}
\end{figure}

\begin{figure}[!h]
  \centering
   \includegraphics[width=0.8\columnwidth] {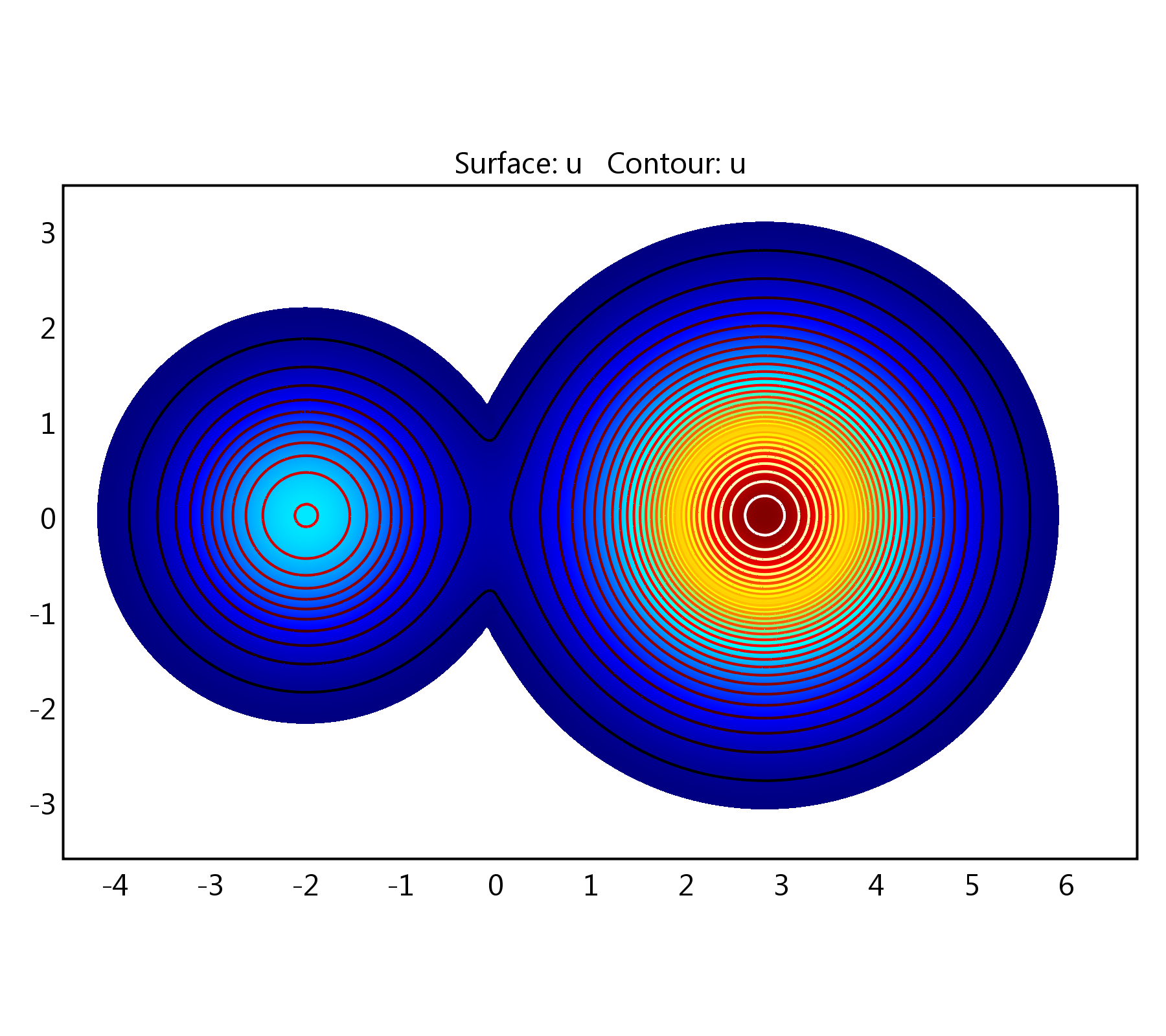}
    \caption{The quadrature domain corresponding to the solution of (\ref{exam2}) for $t=5$.}
    \label{s4}
\end{figure}
\clearpage

\begin{figure}[!h]
  \centering
  \includegraphics[width=0.8\columnwidth] {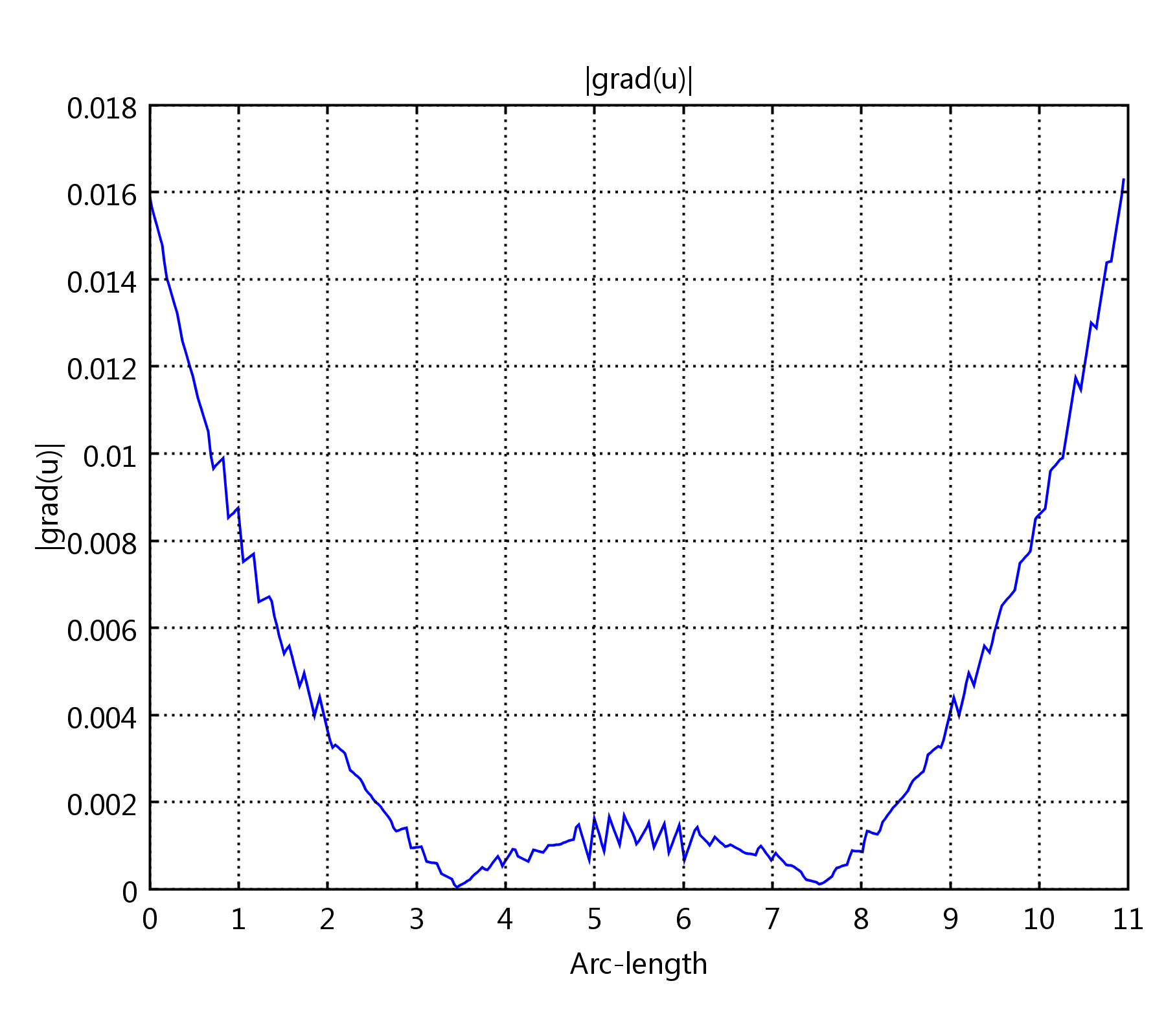}
   \caption{The quantity of $|\nabla u|$ on the boundary of the solution of (\ref{exam2}) for $t=5$.}
    \label{grad5}
\end{figure}

\begin{figure}[!h]
  \centering
   \includegraphics[width=0.8\columnwidth] {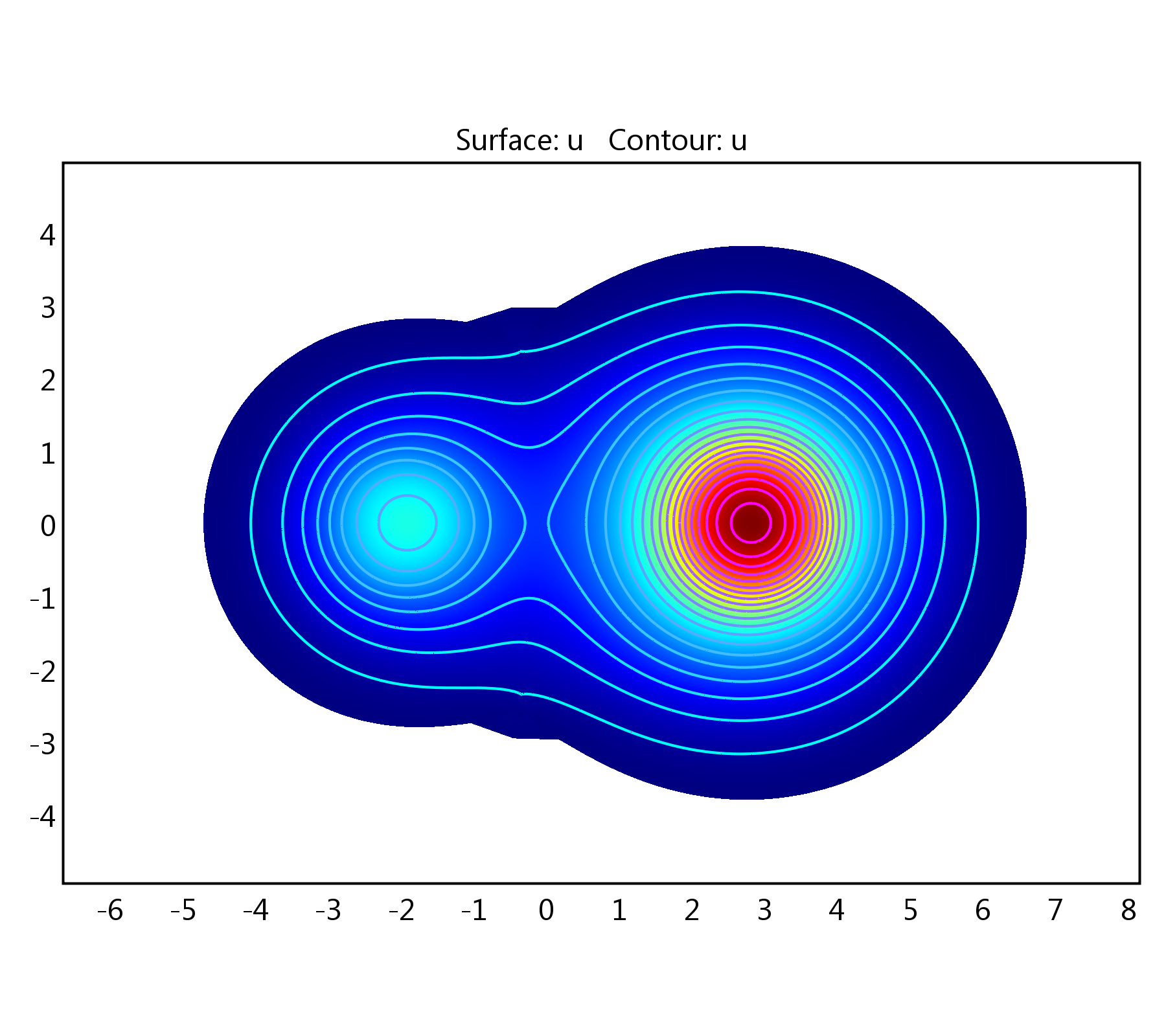}
    \caption{The quadrature domain corresponding to the solution of (\ref{exam2}) for $t=6$.}
    \label{s10}
\end{figure}

\noindent
{\bf  Acknowledgments:} This problem was suggested by Professor Henrik Shahgholian. We wish to  thank him for plentiful discussions and profitable suggestions.

{}
\end{large}
 \end{document}